\documentclass[12pt]{article}
 \usepackage[margin=1in]{geometry} 
\usepackage{amsmath,amsthm,amssymb,amsfonts}
 \usepackage{amssymb}

\usepackage[utf8]{inputenc}
\usepackage[pagewise]{lineno}

\usepackage{hyperref}
 \usepackage{csquotes}
\usepackage[utf8]{inputenc}
\usepackage[english]{babel}
\usepackage{xcolor}

\providecommand{\keywords}[1]
{
  \small	
  \textbf{\textit{Keywords:}} #1
}

\newcommand{\MSC}[1]{%
  \small
  \textbf{\textit{Mathematics Subject Classification:}} #1
}
\title{Absence of blow-up in a fully parabolic chemotaxis system with weak singular sensitivity and logistic damping in dimension two }
\author{
    Minh Le\thanks{The Institute of Theoretical Sciences, Westlake University, China \texttt{(leminh@westlake.edu.cn)}} }
\date{}

\begin{document}
\maketitle

\begin{abstract}
It is shown in this paper that blow-up does not occur in the following chemotaxis system under homogeneous Neumann boundary conditions in a smooth, open, bounded domain \(\Omega \subset \mathbb{R}^2\):  
\begin{equation*} 
      \begin{cases}
          u_t = \Delta u - \chi \nabla \cdot \left( \frac{u}{v^k} \nabla v \right) + ru - \mu u^2,  \qquad &\text{in } \Omega \times (0,T_{\rm max}), \\
         v_t =  \Delta v - \alpha v + \beta u, \qquad &\text{in } \Omega \times (0,T_{\rm max}),  
      \end{cases}
\end{equation*}  
where \( k \in  (0,1) \), and \(\chi, r, \mu, \alpha, \beta \) are positive parameters. Known results have already established the same conclusion for the parabolic-elliptic case. Here, we complement these findings by extending the result to the fully parabolic case.  

\end{abstract}
\\
\MSC{35B35, 35K45, 35K55, 92C15, 92C17}

\numberwithin{equation}{section}
\newtheorem{theorem}{Theorem}[section]
\newtheorem{lemma}[theorem]{Lemma}
\newtheorem{remark}{Remark}[section]
\newtheorem{Prop}{Proposition}[section]
\newtheorem{Def}{Definition}[section]
\newtheorem{Corollary}{Corollary}[theorem]
\allowdisplaybreaks

\section{Introduction}
Cells or microorganisms can direct their movements toward increasing concentrations of a signal they secrete themselves, a phenomenon known as chemotaxis. This process plays a crucial role in biology, as understanding how microorganisms move allows us to predict their pattern formation. In the 1970s, thanks to the pioneering work of Keller and Segel \cite{Keller, Keller2}, this phenomenon was successfully modeled using systems of partial differential equations. Chemotaxis models have attracted significant attention from the mathematical community, not only due to its practical applications but also because of its intriguing mathematical properties. For instance, in the simplest form of the chemotaxis model, known as the Keller-Segel system, there is a phenomenon referred to as critical mass. Specifically, if the total mass of the population exceeds a certain threshold, solutions blow up in finite or infinite time \cite{Nagai1, Nagai2, Nagai3, Nagai4}, whereas if the mass is below this threshold, solutions remain globally bounded \cite{Dolbeault, Dolbeault1, NSY}. Consequently, understanding the conditions that lead to blow-up versus global boundedness is one of the most fundamental topics in the study of chemotaxis systems. In this paper, we investigate whether blow-up can occur for the following system in a smooth bounded domain $\Omega \subset \mathbb{R}^2$:

 \begin{equation} \label{1}
      \begin{cases}
          u_t = \Delta u - \chi \nabla \cdot ( \frac{u}{v^k} \nabla v) +ru -\mu u^2,  \qquad &\text{in } \Omega \times (0,T_{\rm max}), \\
        \kappa v_t =  \Delta v - \alpha v +\beta u, \qquad &\text{in } \Omega \times (0,T_{\rm max}),  \\
          u(x,0) =u_0(x ),  \qquad  \kappa v(x,0)=  \kappa v_0(x), \qquad &\text{in } \Omega,
      \end{cases}
\end{equation}
where $\kappa \in \left \{ 0,1 \right \}$, $\chi>0$, $r >0$, $\mu>0$, $k \in(0,1)$, $T_{\rm max} \in (0, \infty]$ is the maximal existence time, and 
\begin{equation} \label{initial}
    \begin{cases}
        u_0 &\in C^{0} (\bar{\Omega}) \text{ is nonnegative with } \int_\Omega u_0 >0  \\
      v_0 &\in W^{1, \infty}(\Omega), \quad v_0> 0 \quad \text{in }\Omega.
    \end{cases}
\end{equation}
The system \eqref{1} is endowed with the homogeneous Neumann boundary conditions:
\begin{align} \label{bdry}
    \frac{\partial u}{\partial \nu}= 0, \qquad \frac{\partial v}{\partial \nu}=0, \qquad \text{on } \partial \Omega \times (0,T_{\rm max}).
\end{align}
Let us briefly recall some known results related to this system as well as address the main question that we are pursuing in this paper.\\

\textbf{Singular Sensitivity:}
The system \eqref{1}, when $k=\kappa=1$ and $r=\mu =0$, was originally introduced in \cite{Keller+Segel}, where the term $\frac{\chi}{v}$ with $\chi >0$ represents the singular sensitivity determined according to the Weber-Fechner law. One of the first results concerning the blow-up and global existence of solutions was established in \cite{Nagai+Senba}. The authors proved that when $\kappa=0$ and the initial data is radially symmetric, solutions remain globally bounded in $\Omega= B(0,L) \subset \mathbb{R}^n$, where $L>0$ and $n \geq 2$, provided that $\chi< \frac{2}{(n-2)_+}$, while blow-up occurs when $\chi> \frac{2n}{n-2}$ for $n \geq 3$.  Later, it was shown in \cite{Biler-1999} that finite-time blow-up does not occur when $\chi< \frac{2}{n}$ for $n \geq 2$ in the parabolic-elliptic case ($\kappa=0$) and when $\chi \leq 1$ for $n=2$ in the fully parabolic case ($\kappa=1$) for arbitrary initial conditions. This condition on $\chi$ when $\kappa=0$ was later proven in \cite{KMT} to be sufficient to ensure the global boundedness of solutions.  For the fully parabolic case in arbitrary dimensions, \cite{Winkler+2011} demonstrated that solutions exist globally in time under the condition $\chi< \sqrt{\frac{2}{n}}$ for $n \geq 2$, though the question of whether solutions remain globally bounded was left open. Subsequently, in \cite{Fujie}, it was confirmed that solutions are indeed globally bounded under the same restriction on $\chi$, leading to another question: whether $\sqrt{\frac{2}{n}}$ is the optimal bound for $\chi$. In \cite{Lankeit_2016}, a partial answer was provided: the author showed that $\chi< 1.015$ can still prevent blow-up for $n=2$.  
 \\
 
\textbf{Singular Sensitivity with Logistic Source:}
In \cite{MKTAM}, with the presence of the logistic source, $ru - \mu u^2$, when $\kappa=1$, the authors proved that solutions exist globally in time for $n = 2$. In \cite{FWY}, the authors studied a similar problem but for the parabolic-elliptic case in two dimensions. They showed that solutions exist globally for any positive parameters $r, \mu$, and that solutions are bounded under the restriction $r > \frac{\chi^2}{4}$ when $\chi \in (0,2)$ and $r > \chi - 1$ when $\chi \geq 2$. The same boundedness result was also obtained in \cite{Zhao+Zheng} for the fully parabolic case when $n=2$ with the same restricting for $r$ and $\chi$. In higher dimensions, for the parabolic-elliptic case, it was shown in \cite{Kurt+Shen} that when $\mu$ is sufficiently large, solutions exist globally in time. Moreover, solutions are globally bounded when $r$ is sufficiently large and the initial condition is not small. For the fully parabolic case, it was proven in \cite{MWS} in 2019 that the logistic damping term, $-\mu u^\gamma$ with $\gamma>2$, can prevent the occurrence of finite-time blow-up solutions, leaving the question of whether solutions remain bounded or blow up at infinity.\\

\textbf{Weak Singular Sensitivity with Logistic Source:} The weak singular sensitivity, given by $\frac{\chi }{v^k}$ with $k \in (0,1)$, was first studied in \cite{Nagai+Senba} without the logistic source and with $\kappa=0$. The authors demonstrated that solutions remain globally bounded in two dimensions under the assumption of radially symmetric initial data. In \cite{zh}, the author investigated weak singular sensitivity with a logistic source in a parabolic-elliptic system and showed that solutions remain globally bounded when $\mu$ is sufficiently large in two dimensions. Notably, in that work, establishing a positive lower bound for $v$ was not necessary to ensure global boundedness. More recently, this large-$\mu$ assumption was shown to be removable in \cite{Minh6}. Furthermore, a sub-logistic source of the form $-\frac{\mu u}{\ln^\gamma(u+e)}$ with $\gamma \in (0,1)$ was proven to be strong enough to prevent blow-up. Given these findings, it is natural to ask a similar question:  
\textit{"Can a logistic source prevent blow-up in a fully parabolic system with weak singular sensitivity?"}\\

In this paper, we give a partial answer to this question by showing that the logistic source is, in fact, strong enough to prevent blow-up in dimension two. To be more precise, our main result is as follows: 

\begin{theorem} \label{thm1}
For  $\kappa=1$, $\chi>0$, \( r > 0 \), \( \mu > 0 \), \( k \in (0,1) \), \( \alpha > 0 \), and \( \beta > 0 \), the system \eqref{1}, subject to the initial condition \eqref{initial} and the boundary conditions \eqref{bdry}, admits a unique solution \( (u,v) \) such that  
\[
u, v \in C^0 \left(\bar{\Omega} \times [0,\infty)\right) \cap C^{2,1} \left(\bar{\Omega} \times (0,\infty)\right).
\]  
Moreover, \( u \) and \( v \) are strictly positive in \( \bar{\Omega} \times (0, \infty) \), and \( u \) is uniformly bounded in \( \Omega \times (0, \infty) \).
\end{theorem}
\begin{remark}
   Our method does not apply if we replace the logistic source, $-\mu u^2$, with a sub-logistic one of the form $-\frac{\mu u^2}{\ln^\gamma(u + e)}$, where $\gamma \in (0,1)$.
\end{remark}

\textbf{Difficulties and Resolutions:} The presence of a logistic source in the system \eqref{1} can destroy the mass conservation property of the solutions, consequently raising the possibility that $v$ approaches $0$ as $t \to \infty$. In this paper, we offer an approach to overcome this obstacle by considering the energy functional  
\[
y(t)= \int_\Omega u\ln u - \lambda \int_\Omega u \ln v + \frac{1}{2}\int_\Omega |\nabla v|^2,
\]
where $\lambda>0$ is sufficiently small. Using this, we obtain an $L\ln L$ bound for $u$ without relying on a lower bound for $v$. 

However, transitioning from this bound to $L^p$ bounds for $p>1$ presents another challenge: $v$ can still be arbitrarily close to $0$. To address this, we consider the energy functional  
\[
z(t) = \int_\Omega u^p v^{-q} + \int_\Omega u^p + \int_\Omega |\nabla v|^{2p},
\]
where $0<q<p-1$, which allows us to absorb the singularity of $v$ near $0$ into the diffusion.

\textbf{Outline of the paper:} In Section \ref{s2}, we establish the local existence of solutions as well as some inequalities that will be frequently used in the subsequent sections. In the next section, we derive several a priori estimates, including an $L\ln L$ bound and $L^p$ bounds for $p>1$. Finally, we prove the main result in Section \ref{s4}.

\section{Preliminaries} \label{s2}
In this section, we establish the local well-posedness of solutions to the system \eqref{1}, recall a parabolic regularity result in Sobolev spaces, and introduce some useful inequalities for later applications in the subsequent sections. We begin with the following lemma:
\begin{lemma} \label{local-exist}
    Let $\kappa=1$, $\chi>0$, \( r > 0 \), \( \mu > 0 \), \( k \in (0,1) \), \( \alpha > 0 \), and \( \beta > 0 \), and assume that the initial condition \eqref{initial} holds. Then, there exists a maximal existence time \( T_{\rm max} \in (0, \infty] \) and a unique pair of functions \( u \) and \( v \) satisfying  
    \[
    u, v \in C^0 \left(\bar{\Omega} \times [0,T_{\rm max})\right) \cap C^{2,1} \left(\bar{\Omega} \times (0,T_{\rm max})\right),
    \]  
    which solve the system \eqref{1} with the boundary conditions \eqref{bdry} in the classical sense. Moreover, both \( u \) and \( v \) remain strictly positive in \( \bar{\Omega} \times (0,T_{\rm max}) \). If \( T_{\rm max} < \infty \), then  
    \begin{align} \label{local-exist-1}
        \limsup_{t \to T_{\rm max}} \| u(\cdot,t) \|_{L^\infty(\Omega)} = \infty.
    \end{align}  
\end{lemma}
\begin{proof}
The proof is based on a standard fixed-point argument in Banach spaces; for a detailed exposition, we refer readers to \cite{Zhao+Zheng}[Lemma 2.2] to avoid redundancy.
\end{proof}

Henceforth, we denote $(u,v)$ as the unique solution of system \eqref{1} in $\Omega \times (0, T_{\rm max})$, where $T_{\rm max} \in (0, \infty]$, as established in the preceding lemma. The following lemma, a parabolic regularity result in Sobolev spaces, will be applied later in Lemma \ref{Lp-v} and in the proof of Theorem \ref{1}.

\begin{lemma} \label{C52.Para-Reg}
Let $\Omega \subset \mathbb{R}^2$ be an open bounded domain with smooth boundary. Assume that $p\geq 1$ and $q \geq 1$ satisfy 
\begin{equation*}
    \begin{cases}
     q &< \frac{2p}{2-p},  \qquad \text{when } p<2,\\
     q &< \infty, \qquad \text{when } p=2,\\
      q &= \infty, \qquad \text{when } p>2.\\
     \end{cases}
\end{equation*}
Assuming $V_0 \in W^{1,q}(\Omega)$ and $V$ is a classical solution to the following system
\begin{equation}\label{C52.parabolic-equation}
    \begin{cases}
     V_t = \Delta V  - a V + f &\text{in } \Omega \times (0,T), \\ 
\frac{\partial V}{\partial \nu} =  0 & \text{on }\partial \Omega \times (0,T),\\ 
 V(\cdot,0)=V_0   & \text{in } \Omega
    \end{cases}
\end{equation}
where $a>0$ and $T\in (0,\infty]$. If $f \in L^\infty \left ( (0,T);L^p(\Omega) \right ) $, then $V  \in L^\infty \left ( (0,T);W^{1,q}(\Omega) \right )$.
\end{lemma}

\begin{proof}
The proof relies on standard \( L^p \)-\( L^q \) estimates for the Neumann heat kernel; for a detailed proof, we refer readers to \cite{Winkler-Horstmann}[Lemma 4.1].
\end{proof}

The following lemma provides a useful differential inequality, which will be applied later in Lemma \ref{gradv}.

\begin{lemma} \label{ODI}
    Assume that nonnegative function $y \in C^1([0,T))$ where $T \in (0, \infty]$ satisfies 
    \begin{align*}
         \int_t ^{t+\tau} y(s) \,ds  \leq L_1,\qquad \text{for all }t\in (0,T-\tau),
    \end{align*}
    where  $\tau= \min \left \{1, \frac{T}{2} \right \}$ and $L_1>0$, and 
    \begin{align*}
         y'(t) \leq h(t)y(t)+g(t), \qquad \text{for all } t \in (0,T),
    \end{align*}
    where $h$ and $g $ are nonnegative continuous functions in $[0,T)$ such that 
    \begin{align*}
        \int_t ^{t+\tau} h(s) \,ds  \leq L_2, \quad \text{and } \int_t ^{t+\tau} g(s) \,ds  \leq L_3, \quad \text{for all }t\in (0,T-\tau),
    \end{align*}
    where $L_2>0$, and $L_3>0$. Then there exists $C>0$ such that $y(t) \leq C$ for all $t \in (0,T)$.
\end{lemma}
\begin{proof}
    Since $ \sup_{t \in (0, T-\tau)}\int_t ^{t+\tau }y(s)\, ds \leq L_1$, for any $t \in (0, T-\tau)$ there exists $t_0 \in (t, t+\tau)$ depending on $t$ such that 
  \begin{align*}
      y(t_0) \leq \frac{L_1}{\tau}.
  \end{align*}
  Multiplying $e^{- \int_{t_0}^t h(s)\, ds}$ to the inequality $y'(t) \leq h(t)y(t)+g(t)$ and integrating from $0$ to $t_0$ deduces that 
  \begin{align}
      y(t) &\leq y(t_0)e^{  \int_{t_0}^t h(s)\, ds} + \int_{t_0}^t e^{\int_{s}^t h(z)\, dz} g(s)\, ds \notag \\
      &\leq \frac{L_1 e^{L_2}}{\tau}+ L_3 e^{L_2},
  \end{align}
  which completes the proof.
\end{proof}

The next lemma, derived directly from \cite{Winkler_preprint}[Corollary 1.2], provides a modified version of the Gagliardo–Nirenberg interpolation inequality, which will later be used to obtain $L^p$ bounds for $u$ in Lemma \ref{Lp}.

\begin{lemma}\label{C52.ILGN}
Let $\Omega \subset \mathbb{R}^2$ be a bounded domain with a smooth boundary, and suppose $m > 0$ and $\sigma > \xi \geq 0$. Then, for every $\varepsilon > 0$, there exists a constant $C = C(\varepsilon, \xi, \sigma) > 0$ such that the inequality  
\begin{align}\label{C52.ILGN.1}
    \int_\Omega \phi^{m+1} \ln^{\xi}(\phi+e) \,dx 
    &\leq \varepsilon \left( \int_\Omega \phi \ln^{\sigma}(\phi+e) \,dx \right) 
    \left( \int_\Omega |\nabla \phi^{\frac{m}{2}}|^2 \,dx \right) \notag \\
   & \quad + \varepsilon \left( \int_\Omega \phi \,dx \right)^m 
    \left( \int_\Omega \phi \ln^{\sigma}(\phi+e) \,dx \right) + C,
\end{align}
holds for any nonnegative function $\phi \in C^1(\bar{\Omega})$.
\end{lemma}
\begin{proof}
   For detailed proofs, we refer the reader to Lemmas 2.4 and 2.5 in \cite{Minh5}.
\end{proof}

\section{A Priori Estimates}\label{s3}
In this section, we present some important estimates for solutions, especially we derive an \( L \ln L \) bound and \( L^p \) bounds for \( u \). Let us begin with an $L^1$ bound as in the following lemma:
\begin{lemma} \label{L1-est}
    There exist $m>0$ and $C>0$ such that
    \begin{align*}
        \int_\Omega u(\cdot,t) \leq m, \qquad \text{for all }t \in (0, T_{\rm max})
    \end{align*}
    and 
    \begin{align*}
      \int_t ^{t+\tau}   \int_\Omega u^2(\cdot,t) \leq C, \qquad \text{for all }t \in (0, T_{\rm max}-\tau), 
    \end{align*}
    where $\tau = \min \left \{ \frac{T_{\rm max}}{2},1 \right \}$.
\end{lemma}
\begin{proof}
    Integrating the first equation of \eqref{1} over $\Omega$ and applying Holder's inequality yields that
    \begin{align} \label{L1-est.1}
        \frac{d}{dt} \int_\Omega  u(\cdot,t)&= r  \int_\Omega  u(\cdot,t) - \mu  \int_\Omega  u^2(\cdot,t) \notag \\
        &\leq r  \int_\Omega  u(\cdot,t) - \frac{\mu}{|\Omega|} \left ( \int_\Omega  u(\cdot,t) \right )^2,
    \end{align}
    for all $t \in (0,T_{\rm max})$. By standard comparison principle, we obtain that 
    \[
     \int_\Omega  u(\cdot,t) \leq m,
    \]
    for all $t \in (0,T_{\rm max})$ where $m=\max \left \{ \frac{r|\Omega|}{\mu}, \int_\Omega u_0. \right \}$. Now, noting that 
    \begin{align*}
     \int_t ^{t+\tau}\int_\Omega u^2 &=  \frac{r}{\mu} \int_t ^{t+\tau}\int_\Omega u - \frac{1}{\mu}\int_\Omega u(\cdot,t+\tau) + \frac{1}{\mu}\int_\Omega u(\cdot,t) \notag \\
        &\leq \frac{m(r\tau+1)}{\mu}, 
    \end{align*}
    which, together with \eqref{L1-est.1} completes the proof.
    
\end{proof}
With an \( L^1 \) bound for \( u \), we can establish \( L^p \) bounds for \( v \) for any \( p \geq 1 \), thanks to the parabolic regularity result in Sobolev spaces, as established in Lemma \ref{C52.parabolic-equation}.

\begin{lemma} \label{Lp-v}
    For any $p\geq 1$, there exists $C=C(p)>0$ such that 
    \begin{align}
        \int_\Omega v^p(\cdot,t) \leq C, \qquad \text{for all }t\in (0,T_{\rm max}).
    \end{align}
\end{lemma}
\begin{proof}
    Integrating the second equation of \eqref{1} over $\Omega$ and applying Lemma \ref{L1-est} yields
    \begin{align*}
        \frac{d}{dt}\int_\Omega v + \alpha \int_\Omega v &= \beta \int_\Omega u \notag \\
        &\leq \beta m,
    \end{align*}
    where $m $ is given in \ref{L1-est}. Therefore, applying Gronwall's inequality to this implies that 
    \begin{align}
        \sup_{t\in (0,T_{\rm max})} \int_\Omega v(\cdot,t) \leq \max \left \{ \int_\Omega v_0, \frac{\beta m}{\alpha} \right \}.
    \end{align}
    Since $u \in L^\infty \left ( (0,T_{\rm max});L^1(\Omega) \right ) $, we apply Lemma \ref{C52.parabolic-equation} to obtain that $v \in L^\infty \left ( (0,T_{\rm max});W^{1,q}(\Omega) \right )$ for any $q \in [1,2)$. Now, applying Sobolev's inequality deduces that 
    \begin{align*}
        \int_\Omega v^p \leq c_1  \left ( \int_\Omega |\nabla v|^{\frac{2p}{2+p}} \right )^{\frac{p+2}{2}}+c_1 \left ( \int_\Omega v \right )^p \leq c_2, \quad \text{for all }t \in (0,T_{\rm max}),
    \end{align*}
     where $c_1>0$ and $c_2>0$. The proof is now complete.
   
\end{proof}
The following lemma provides an \( L^2 \) bound for the gradient of \( v \), which will be applied later in Lemma \ref{L1}.

\begin{lemma} \label{gradv}
    There exists $C>0$ such that
    \begin{align*}
        \int_\Omega |\nabla v(\cdot,t)|^2 \leq C, \qquad \text{for all }t \in (0, T_{\rm max}).
    \end{align*}
\end{lemma}
\begin{proof}
    Multiplying the second equation of \eqref{1} by $v$ and applying Young's inequality yields 
    \begin{align*}
        \frac{1}{2}\frac{d}{dt} \int_\Omega v^2 &= - \int_\Omega |\nabla v|^2 -\alpha \int_\Omega v^2 + \beta \int_\Omega uv \notag \\
        &\leq - \int_\Omega |\nabla v|^2 + \frac{\beta^2}{4 \alpha} \int_\Omega u^2.
    \end{align*}
    This, together with Lemma \ref{L1-est} and Lemma \ref{Lp-v} leads to 
    \begin{align} \label{gradv.1}
         \int_t ^{t+\tau} \int_\Omega |\nabla v(\cdot,s)|^2\, ds \leq \frac{\beta^2}{4 \alpha} \int_t ^{t+\tau} \int_\Omega u^2(\cdot,s)\, ds + \frac{1}{2} \int_\Omega v^2(\cdot,t) \leq c_1, \quad \text{for all } t\in (0, T_{\rm max-\tau}).
    \end{align}
    By differentiating the functional $y(t):= \frac{1}{2}\int_\Omega |\nabla v(\cdot,t)|^2$ in time, we obtain that
    \begin{align*}
        y'(t)+2\alpha y(t) &= - \int_\Omega (\Delta v)^2  -\beta \int_\Omega u \Delta v \notag \\
        &\leq  \frac{\beta^2}{4} \int_\Omega u^2.
    \end{align*}
    This, in conjunction with Lemma \ref{L1-est}, \eqref{gradv.1} and Lemma \ref{ODI} entails that $\sup_{t \in (0,T_{\rm max})} y(t) \leq c_2$ for some $c_2>0$, which completes the proof.
    
\end{proof}

The following lemma, though seemingly simple, is highly useful and serves as the key idea for establishing an \( L \ln L \) estimate for \( u \).

\begin{lemma} \label{M}
    Let $l(t):= -\int_\Omega u \ln v$ then
    \begin{align*}
           l'(t)  &\leq  rl(t)+2\int_\Omega \frac{|\nabla u|^2}{u}+ \alpha \int_\Omega u + \mu \int_\Omega u^2 \ln v  \notag \\
       &\quad -\chi \int_\Omega \frac{ u }{v^{1+k}}|\nabla v|^2 - \frac{1}{2}\int_\Omega \frac{ u }{v^{2}}|\nabla v|^2 -\beta \int_\Omega \frac{u^2}{v}.
    \end{align*}
\end{lemma}

\begin{proof}
    Direct calculations shows us that
    \begin{align*}
        l'(t)&= -\int_\Omega u_t \ln v- \int_\Omega \frac{u}{v} v_t \notag \\
        &= - \int_\Omega  \left ( \Delta u  -\chi \nabla \cdot \left (u \frac{\nabla v}{v^k} \right )  +ru -\mu u^2 \right ) \ln v \notag \\
        &\quad - \int_\Omega \frac{u}{v } \left ( \Delta v- \alpha v +\beta u \right ) \notag \\
        &=2 \int_\Omega \nabla u \cdot \frac{ \nabla v }{v} -\chi \int_\Omega \frac{ u }{v^{1+k}}|\nabla v|^2 - \int_\Omega \frac{ u }{v^{2}}|\nabla v|^2 + \alpha \int_\Omega u - \beta \int_\Omega \frac{u^2 }{v} \notag \\
        &\quad +rl(t) + \mu \int_\Omega u^2 \ln v
    \end{align*}
    This, together with Young's inequality deduces that
    \begin{align*}
       l'(t) - r l(t)  &\leq  2\int_\Omega \frac{|\nabla u|^2}{u}+ \alpha \int_\Omega u + \mu \int_\Omega u^2 \ln v  \notag \\
       &\quad -\chi \int_\Omega \frac{ u }{v^{1+k}}|\nabla v|^2 - \frac{1}{2} \int_\Omega \frac{ u }{v^{2}}|\nabla v|^2 -\beta \int_\Omega \frac{u^2}{v}.
    \end{align*}
    The proof is now complete.
\end{proof}
We are now ready to establish the first key ingredient in proving the main result, namely, an \( L \ln L \) bound for \( u \).

\begin{lemma}\label{L1}
    There exists $C>0$ such that
   \begin{align*}
        \int_\Omega u(\cdot,t ) \ln u(\cdot,t)  \leq C, \qquad \text{for all }t\in (0,T_{\rm max}).
   \end{align*}
\end{lemma}
\begin{proof}
    Differentiating the function $y(t):=   \int_\Omega u(\cdot,t ) \ln u(\cdot,t) + \frac{1}{2}\int_\Omega |\nabla v(\cdot.t)|^2$ in time yields
    \begin{align} \label{L1.1}
        y'(t)&= \int_\Omega (\ln u+1)u_t + \int_\Omega \nabla v \cdot  \nabla v_t \notag \\
        &:= I+J.
    \end{align}
    Making use of the first equation of \eqref{1} and integration by parts implies that 
    \begin{align} \label{L1.2}
        I &= \int_\Omega (\ln u+1) (\Delta u - \chi \nabla \cdot \left ( u \frac{\nabla v}{v^k} \right ) +ru -\mu u^2 ) \notag \\
        &= - \int_\Omega \frac{|\nabla u|^2}{u} + \chi \int_\Omega \frac{\nabla u \cdot \nabla v}{v^k} + \int_\Omega  (ru - \mu u^2) (\ln u +1).
    \end{align}
    By applying Young's inequality with $\varepsilon>0$, which will be determined later, we obtain
    \begin{align} \label{L1.3}
        \chi  \int_\Omega \frac{\nabla u \cdot \nabla v}{v^k} &\leq \varepsilon \int_\Omega \frac{|\nabla u|^2}{u} + \frac{\chi^2}{4 \varepsilon} \int_\Omega u \frac{|\nabla v|^2}{v^{2k}} \notag \\
        &\leq \varepsilon \int_\Omega \frac{|\nabla u|^2}{u} + \frac{\varepsilon}{2} \int_\Omega u \frac{|\nabla v|^2}{v^{2}}+c_1 \int_\Omega u|\nabla v|^2 \notag \\
        &\leq  \varepsilon \int_\Omega \frac{|\nabla u|^2}{u} + \frac{\varepsilon}{2} \int_\Omega u \frac{|\nabla v|^2}{v^{2}} + \varepsilon \int_\Omega |\nabla v|^4 +c_2 \int_\Omega u^2
    \end{align}
    where $c_1>0$ and $c_2>0$. Applying Gagliardo–Nirenberg interpolation inequality and Lemma \ref{gradv} deduces that
    \begin{align}\label{L1.4}
      \varepsilon \int_\Omega |\nabla v|^4 &\leq  c_3 \varepsilon \int_\Omega (\Delta v)^2 \int_\Omega |\nabla v|^2 +c_3 \varepsilon \left ( \int_\Omega |\nabla v|^2 \right )^2 \notag \\
      &\leq \varepsilon c_4 \int_\Omega (\Delta v)^2 +c_5,
    \end{align}
    where $c_3$, $c_4$, and $c_5$ are positive constants. Using integration by parts and Lemma \ref{gradv} implies that
    \begin{align}\label{L1.5}
      J + \frac{1}{2}\int_\Omega |\nabla v|^2 &= - \int_\Omega (\Delta v)^2 +  \left ( \frac{1}{2}-\alpha \right ) \int_\Omega |\nabla v|^2 -\beta \int_\Omega u \Delta v \notag \\
        &\leq -\frac{1}{2}\int_\Omega (\Delta v)^2+\frac{\beta^2}{2}\int_\Omega u^2 +c_6,
    \end{align}
    where $c_6>0$. One can verify that
    \begin{align}\label{L1.6}
        \left ( c_2 + \frac{\beta^2}{2} \right ) \int_\Omega u^2 + \int_\Omega u \ln u + \int_\Omega (ru-\mu u^2)(\ln u+1) \leq - \frac{\mu}{2} \int_\Omega u^2 +c_7,
    \end{align}
    for some $c_7>0$. Collecting from \eqref{L1.1} to \eqref{L1.6} implies that
    \begin{align}\label{L1.7}
        y'(t)+y(t) &\leq (\varepsilon-1) \int_\Omega \frac{|\nabla u|^2}{u} + \frac{\varepsilon}{2} \int_\Omega u \frac{|\nabla v|^2}{v^2} + \left ( \varepsilon c_4- \frac{1}{2} \right ) \int_\Omega (\Delta v)^2 -\frac{\mu}{2}\int_\Omega u^2 \ln u+c_8,
    \end{align}
    where $c_8=c_5+c_6+c_7$. Applying Lemma \ref{M} entails that 
    \begin{align}\label{L1.8}
       \varepsilon l'(t)+ \varepsilon l(t) &\leq 2\varepsilon \int_\Omega \frac{|\nabla u|^2}{u} + \alpha \varepsilon m + \mu \varepsilon \int_\Omega u^2 \ln v  \notag \\
        &\quad -\frac{\varepsilon}{2} \int_\Omega u \frac{|\nabla v|^2}{v^2} -\beta \varepsilon \int_\Omega \frac{u^2}{v} - \varepsilon(r+1)\int_\Omega u \ln v,
    \end{align}
    where $m$ is given in Lemma \ref{L1-est}. By using an elementary inequality that $xy \leq x\ln x+e^{y-1}$ for all $x>0$ and $y>0$ and Lemma \ref{Lp-v}, we obtain that
    \begin{align}\label{L1.9}
        \mu \varepsilon \int_\Omega u^2 \ln v &= \frac{\mu}{4} \int_\Omega u^2 \ln v^{4\varepsilon}  \notag  \\
        &\leq \frac{\mu}{2} \int_\Omega u^2 \ln u + \frac{\mu}{4e} \int_\Omega v^{4\varepsilon} \notag  \\
        &\leq  \frac{\mu}{2} \int_\Omega u^2 \ln u + c_9,
        \end{align}
        where $c_9>0$. Applying Young's inequality and Lemma \ref{Lp-v} entails that  
    \begin{align}\label{L1.10}
        -(r+1)\varepsilon \int_\Omega u \ln v &\leq \beta \varepsilon  \int_\Omega \frac{u^2}{v}+\frac{(r+1)^2 \varepsilon}{4\beta } \int_\Omega v \ln^2 v \notag \\
        &\leq \beta \varepsilon  \int_\Omega \frac{u^2}{v}+c_{10},
    \end{align}
    where $c_{10}>0$. Combining from \eqref{L1.7} to \eqref{L1.10}  entails that
    \begin{align}\label{L1.11}
        y'(t) +y(t)+\varepsilon l'(t)+\varepsilon l(t) &\leq (-1+3\varepsilon) \int_\Omega \frac{|\nabla u|^2}{u} 
       +\left (\varepsilon c_4 -\frac{1}{2} \right )\int_\Omega (\Delta v)^2+ c_{11},
    \end{align}
    where $c_{11}=\alpha \varepsilon m+ c_8+c_9+c_{10}$. Choosing $\varepsilon = \min \left \{ \frac{1}{3}, \frac{1}{2c_4} \right \}$ deduces that
    \begin{align}\label{L1.12}
         y'(t) +y(t)+\varepsilon l'(t)+\varepsilon l(t)  &\leq   c_{11} .
    \end{align}
    Applying Gronwall's inequality to this entails that 
    \begin{align*}
        \int_\Omega u \ln u - \varepsilon \int_\Omega u \ln v + \frac{1}{2} \int_\Omega |\nabla v|^2 \leq c_{12},
    \end{align*}
    for some $c_{12}>0$.  This, together with the inequality that $xy \leq x\ln x+e^{y-1}$ for all $x>0$ and $y>0$ implies 
    \begin{align*}
         \int_\Omega u \ln u &\leq \varepsilon \int_\Omega u \ln v +c_{12} \\
          &\leq \varepsilon \int_\Omega u \ln u+ \frac{\varepsilon}{e} \int_\Omega v+c_{12} \\
          &\leq  \varepsilon \int_\Omega u \ln u+ c_{13}.
    \end{align*}
 Note that $\varepsilon < 1$; thus, the lemma follows.
\end{proof}
Next, we derive the following estimate for the gradient of \( v \), which will assist us in establishing an \( L^p \) bound for \( u \).

\begin{lemma} \label{C5.l}
    For any $p >1$, there exist positive constants $C_1,C_2,C_3$ depending only on $p$ such that 
    \begin{align} \label{C5.l-1}
        \frac{d}{dt}\int_\Omega |\nabla v|^{2p}+ \int_\Omega |\nabla v|^{2p} \leq -C_1\int_\Omega \left | \nabla |\nabla v|^p \right |^2+ C_2 \int_\Omega u^2|\nabla v|^{2p-2} +C_3 \int_\Omega |\nabla v|^{2p}
    \end{align}
\end{lemma}
\begin{proof}
 The lemma can be proved using a standard testing argument, as established, for instance, in \cite{Minh4}[Lemma 4.2].
\end{proof}

Now, we derive the second key ingredient for proving the main result, as stated in the following lemma.

\begin{lemma} \label{Lp}
    There exists $C>0$ such that
    \begin{align*}
   \int_\Omega u^p+ \int_\Omega |\nabla v|^{2p}+ \int_\Omega u^pv^{-q}\leq C, \qquad \text{for all } t\in (0,T_{\rm max}).     
    \end{align*}  
    for $p>1$ and $0<q<p-1$.
\end{lemma}

\begin{proof}
    Differentiating the following functional yields
    \begin{align} \label{Lp.1}
        \frac{d}{dt} \int_\Omega u^p v^{-q} &= p \int_\Omega u^{p-1}v^{-q}u_t -q \int_\Omega u^p v^{-q-1}v_t \notag \\
        &= p \int_\Omega u^{p-1}v^{-q} \left ( \Delta u -\chi \nabla \cdot \left ( uv^{-k} \nabla v \right ) +ru -\mu u^2 \right ) \notag \\
        &\quad -q \int_\Omega u^p v^{-q-1} \left ( \Delta v- \alpha v +\beta u \right ) \notag \\
        &=- p(p-1)\int_\Omega u^{p-2}v^{-q}|\nabla u|^2 +2pq \int_\Omega u^{p-1}v^{-q-1} \nabla u \cdot \nabla v  \notag\\
        &\quad + p(p-1)\chi \int_\Omega u^{p-1}v^{-q-k} \nabla u \cdot \nabla v -pq\chi \int_\Omega u^pv^{-q-k-1}|\nabla v|^2  \notag \\
        &\quad + (rp+q\alpha) \int_\Omega u^p v^{-q} -\mu p \int_\Omega u^{p+1}v^{-q} -q(q+1)\int_\Omega u^p v^{-q-2}|\nabla v|^2- q\beta \int_\Omega u^{p+1}v^{-q-1},
        \end{align}
        and 
        \begin{align}\label{Lp.2}
            \frac{d}{dt} \int_\Omega u^p = -\frac{4(p-1)}{p}\int_\Omega |\nabla u^\frac{p}{2}|^2 + 2\chi(p-1) \int_\Omega u^\frac{p}{2}v^{-k} \nabla u^\frac{p}{2} \cdot \nabla v + rp \int_\Omega u^p -\mu p \int_\Omega u^{p+1}.
        \end{align}
   Applying Young's inequality with $\varepsilon_1 > 0$, which will be determined later, implies that
    \begin{align}\label{Lp.3}
        2pq \int_\Omega u^{p-1}v^{-q-1} \nabla u \cdot \nabla v \leq (p(p-1)-\varepsilon_1) \int_\Omega u^{p-2}v^{-q} |\nabla u|^2+ \frac{p^2q^2}{p(p-1)-\varepsilon_1} \int_\Omega u^p v^{-q-2}|\nabla v|^2,
    \end{align}
    and 
    \begin{align}\label{Lp.4}
        p(p-1)\chi \int_\Omega u^{p-1}v^{-q-k} \nabla u \cdot \nabla v &\leq \varepsilon_1 \int_\Omega u^{p-2} v^{-q}|\nabla u|^2 + \frac{p^2(p-1)^2\chi^2}{4\varepsilon_1}\int_\Omega u^p v^{-q-2k}|\nabla v|^2 \notag \\
        &\leq \varepsilon_1 \int_\Omega u^{p-2} v^{-q}|\nabla u|^2 + \varepsilon_1 \int_\Omega u^p v^{-q-2}|\nabla v|^2 +c_1 \int_\Omega u^p|\nabla v|^2,
    \end{align}
    where $c_1=c_1(\varepsilon_1)>0$, and
    \begin{align}\label{Lp.5}
        2\chi(p-1) \int_\Omega u^\frac{p}{2} v^{-k} \nabla u^\frac{p}{2} \cdot \nabla v &\leq \frac{p-1}{p} \int_\Omega |\nabla u^\frac{p}{2}|^2 + p(p-1)\chi^2 \int_\Omega u^p v^{-2k} |\nabla v|^2 \notag\\
        &\leq  \frac{p-1}{p} \int_\Omega |\nabla u^\frac{p}{2}|^2 +\varepsilon_1 \int_\Omega u^p v^{-q-2} |\nabla v|^2 +c_2 \int_\Omega u^p|\nabla v|^2,
    \end{align}
    where $c_2=c_2(\varepsilon_1)>0$.
    Using Young's inequality again and Lemma \ref{Lp-v} entails that 
    \begin{align}\label{Lp.6}
        (1+rp+q\alpha) \int_\Omega u^p v^{-q} &\leq q\beta \int_\Omega u^{p+1}v^{-q-1}+c_3 \int_\Omega v^{p-q} \notag \\
         &\leq q\beta \int_\Omega u^{p+1}v^{-q-1}+c_4,
    \end{align}
    for some $c_3>0$ and $c_4>0$, and
    \begin{align}\label{Lp.7}
        (rp+1) \int_\Omega u^p \leq \mu p \int_\Omega u^{p+1} + c_5,
    \end{align}
    where $c_5>0$. Collecting from \eqref{Lp.1} to \eqref{Lp.7} leads to 
    \begin{align}\label{Lp.8}
         \frac{d}{dt} \int_\Omega u^p v^{-q}+ \int_\Omega u^p v^{-q}+\frac{d}{dt} \int_\Omega u^p +\int_\Omega u^p &\leq \left ( \frac{p^2q^2}{p(p-1)-\varepsilon_1} +2\varepsilon_1 -q(q+1) \right )\int_\Omega u^p v^{-q-2}|\nabla v|^2 \notag\\
         &\quad - \frac{3(p-1)}{p} \int_\Omega |\nabla u^\frac{p}{2}|^2 +c_6 \int_\Omega u^p|\nabla v|^2+c_7,
    \end{align}
    where $c_6=c_1+c_2$, and $c_7=c_4+c_5$. The condition $0<q<p-1$ entails that $\frac{p^2q^2}{p(p-1)}< q(q+1)$, which allows us to choose $\varepsilon_1 $ sufficiently small such that
    \[
    \frac{p^2q^2}{p(p-1)-\varepsilon_1} +2\varepsilon_1 -q(q+1)  \leq 0.
    \]
    This, together with \eqref{Lp.8} infers that 
    \begin{align}\label{Lp.9}
          \frac{d}{dt} \int_\Omega u^p v^{-q}+ \int_\Omega u^p v^{-q}+\frac{d}{dt} \int_\Omega u^p +\int_\Omega u^p &\leq  - \frac{3(p-1)}{p} \int_\Omega |\nabla u^\frac{p}{2}|^2 +c_6 \int_\Omega u^p|\nabla v|^2+c_7.
    \end{align}
    Setting $y(t):= \int_\Omega u^p v^{-q} +\int_\Omega u^p +  \int_\Omega |\nabla v|^{2p}$ and applying Lemma \ref{C5.l} deduces that 
    \begin{align}\label{Lp.10}
       y'(t)+y(t)   &\leq  - \frac{3(p-1)}{p} \int_\Omega |\nabla u^\frac{p}{2}|^2  -c_8\int_\Omega \left | \nabla |\nabla v|^p \right |^2 +c_6 \int_\Omega u^p|\nabla v|^2 \notag\\
         & \quad +c_9 \int_\Omega u^2|\nabla v|^{2p-2} +c_{10} \int_\Omega |\nabla v|^{2p} +c_7,
    \end{align}
    where $c_8, c_9,c_{10}$ are positive constants depending only on $p$. In light of Young's inequality, we obtain that 
    \begin{align}\label{Lp.11}
        c_6 \int_\Omega u^p|\nabla v|^2 +c_9 \int_\Omega u^2|\nabla v|^{2p-2} +c_{10} \int_\Omega |\nabla v|^{2p}+c_7 \leq \varepsilon_2 \int_\Omega |\nabla v|^{2p+2}+ c_{11} \int_\Omega u^{p+1} +c_{12},
    \end{align}
    where  $\varepsilon_2>0$ will be determined later,   $c_{11}>0$ and $c_{12}>0$. Applying Gagliardo–Nirenberg interpolation inequality and Lemma \ref{gradv} yields that
    \begin{align}\label{Lp.12}
        \varepsilon_2 \int_\Omega |\nabla v|^{2p+2} &\leq \varepsilon_2 c_{13}\int_\Omega |\nabla |\nabla v|^p|^2 \int_\Omega |\nabla v|^2 +\varepsilon_2 c_{13} \left ( \int_\Omega |\nabla v|^2 \right )^\frac{p}{2} \notag \\
        &\leq c_8\int_\Omega |\nabla |\nabla v|^p|^2 +c_{14},
    \end{align}
    where $c_{13}>0$, $\varepsilon_2= \frac{c_8}{c_{13}\sup_{t \in (0, T_{\rm max})} \int_\Omega |\nabla v(\cdot,t )|^2}$ and $c_{14}>0$. Lemma \ref{L1} asserts that \[\sup_{t \in (0,T_{\rm max})} \int_\Omega u \ln u <\infty,\] which allows us to apply Lemma \ref{C52.ILGN} with 
    \[
    \varepsilon_3= \frac{p-1}{pc_{11}\sup_{t \in (0,T_{\rm max})}\int_\Omega u \ln (u+e)}
    \]
   to obtain that
    \begin{align}\label{Lp.13}
        c_{11} \int_\Omega u^{p+1} &\leq \varepsilon_3 c_{11} \int_\Omega |\nabla u^\frac{p}{2}|^2 \int_\Omega u \ln (u+e) + \varepsilon_3 c_{11} \left ( \int_\Omega u \right )^p \int_\Omega u \ln (u+e)+c_{15} \notag \\
        &\leq \frac{p-1}{p} \int_\Omega |\nabla u^\frac{p}{2}|^2 +c_{16},
    \end{align}
    where  $c_{15}>0$, and $c_{16}>0$. Collecting from \eqref{Lp.10} to \eqref{Lp.13} implies that 
    \begin{align*}
        y'(t)+y(t) \leq c_{17}, \quad \text{for all }t \in (0,T_{\rm max}),
    \end{align*}
    where $c_{17}=c_7+c_{12}+c_{14}+c_{16}$. Applying Gronwall's inequality to this entails that 
    \begin{align*}
        y(t) \leq \max \left \{\int_\Omega u^p_0 + \int_\Omega u^p_0v^{-q}_0+ \int_\Omega |\nabla v_0|^{2p} , c_{17} \right \}, \qquad \text{for all }t \in (0,T_{\rm max}),
    \end{align*}
    which completes the proof.
\end{proof}
As a consequence of the above lemma, we conclude this section with the following estimate.

\begin{lemma} \label{L3}
    Let $p> \max \left \{ 2, \frac{1}{1-k} \right \}$ then there exists $C=C(p)>0$ such that 
    \begin{align*}
       \int_\Omega \frac{u^p(\cdot,t)|\nabla v(\cdot,t)|^p}{v^{kp}(\cdot,t)} \leq C, \qquad \text{for all }t \in (0,T_{\rm max}).
    \end{align*}
\end{lemma}
\begin{proof}
    From Lemma \ref{Lp}, it follows that $u \in L^\infty \left ((0,T_{\rm max});L^p(\Omega) \right )$ for some $p>2$. This, in conjunction with Lemma \ref{C52.Para-Reg} entails that 
    \[
    \sup_{t \in (0,T_{\rm max})} \left \| \nabla v(\cdot,t) \right \|_{L^\infty(\Omega)}=c_1 < \infty.
    \]
    Moreover, we notice that $kp< p-1$ since $p> \frac{1}{1-k}$, which together with Lemma \ref{Lp} implies that 
    \[
    \sup_{t \in (0, T_{\rm max})} \int_\Omega \frac{u^p(\cdot,t)}{v^{kp}(\cdot,t)} = c_2 <\infty.
    \]
    Therefore, we obtain that 
    \begin{align*}
         \int_\Omega \frac{u^p(\cdot,t)|\nabla v(\cdot,t)|^p}{v^{kp}(\cdot,t)} \leq \sup_{t \in (0,T_{\rm max})} \left \| \nabla v(\cdot,t) \right \|_{L^\infty(\Omega)} \int_\Omega \frac{u^p(\cdot,t)}{v^{kp}(\cdot,t)} \leq c_1^pc_2, \quad \text{for all }t\in (0,T_{\rm max}),
    \end{align*}
    which proves the lemma.
\end{proof}

\section{Proof of the main result} \label{s4}
By applying standard heat semigroup estimates, we can now prove the uniform boundedness of solutions. Here, we follow a similar argument as in the proof of Theorem 0.1 in \cite{Winkler-logistic}.

\begin{proof}[Proof of Theorem \ref{thm1}]
    Since $ru -\mu u^2 \leq \frac{r^2}{4 \mu}$, we obtain that 
    \begin{align}\label{pr.1}
        u(\cdot,t)\leq  e^{(t-t_0) \Delta }u(\cdot,t_0) + \int_{t_0} ^t e^{(t-s)\Delta} \nabla \cdot \left ( u\frac{\nabla v}{v^k} \right )(\cdot,s)\, ds +\int_{t_0}^t e^{(t-s)\Delta }  \frac{r^2}{4\mu } \, ds 
    \end{align}
    where $t \in (0, T_{\rm max})$ and $t_0 = \max \left \{ 0, t-1\right \}$. Applying standard $L^p-L^q$ estimates for $(e^{t\Delta})_{t \geq 0 }$ (see \cite{Winkler-2010}[Lemma 1.3]) implies that 
    \begin{align} \label{pr.2}
        \left \|  e^{(t-t_0) \Delta }u(\cdot,t_0) \right \|_{L^\infty(\Omega)} \leq c_1m (1+(t-t_0)^{-1}), 
    \end{align}
    where $c_1>0$ and $m $ is the constant in Lemma \ref{L1-est}. For $q> \max \left \{ 2, \frac{1}{1-k} \right \}$ and $p \in (q, \infty)$, applying Lemma \ref{L3} entails that
    \begin{align} \label{pr.3}
        \int_{t_0} ^t e^{(t-s)\Delta} \nabla \cdot \left ( u\frac{\nabla v}{v^k} \right )(\cdot,s)\, ds &\leq  \int_{t_0} ^t \left \| e^{(t-s)\Delta} \nabla \cdot \left ( u\frac{\nabla v}{v^k} \right )(\cdot,s) \right \|_{L^\infty(\Omega)} \, ds \notag \\
        &\leq c_2 \int_{t_0} ^t (t-s)^{-\frac{1}{p}} \left \| e^{\frac{t-s}{2}\Delta} \nabla \cdot \left ( u\frac{\nabla v}{v^k} \right )(\cdot,s) \right \|_{L^p(\Omega)} \,ds \notag \\
        &\leq  c_3 \int_{t_0} ^t (t-s)^{-\frac{1}{2}-\frac{1}{q}}  \left \| u(\cdot,s)\frac{\nabla v(\cdot,s)}{v^k(\cdot,s)} \right \|_{L^q(\Omega)}  \, ds \notag \\
        &\leq c_4 (t-t_0)^{\frac{1}{2}-\frac{1}{q}}.
    \end{align}
    Additionally, we have that 
    \begin{align} \label{pr.4}
        \int_{t_0}^t e^{(t-s)\Delta }  \frac{r^2}{4\mu } \, ds  = \frac{r^2}{4\mu }(t-t_0).
    \end{align}
    Collecting from \eqref{pr.1} to \eqref{pr.4} yields
    \begin{align}
        \left \|u(\cdot,t) \right \|_{L^\infty(\Omega)} &\leq c_5 ((t-t_0)^{-1}+(t-t_0)^{\frac{1}{2}-\frac{1}{q}} + (t-t_0)) \notag \\
        &\leq c_6(t^{-1}+1), \quad \text{for all }t \in (0, T_{\rm max}),
    \end{align}
     which further deduces that  $u$ is bounded in $\Omega \times (0,T_{\rm max})$ since $u$ is bounded in $\Omega \times (0,\frac{T_{\rm max}}{2})$ by Lemma \ref{local-exist}. Therefore, by the extensibility property of the solutions as established in Lemma \eqref{local-exist-1}, it follows that $T_{\rm max} = \infty$ and that $(u,v)$ is bounded in $\Omega \times (0, \infty)$. 
\end{proof}

\end{document}